\documentclass[12pt]{amsart}
\usepackage{amsmath,amsthm,amssymb,bbm,dsfont}
\usepackage{graphicx}
\usepackage{subfig}
\usepackage{cite,float}

\numberwithin{equation}{section}
\newtheorem{theorem}{Theorem}[section]
\newtheorem{lemma}[theorem]{Lemma}
\newtheorem{result}[theorem]{Result}

\newtheorem{proposition}[theorem]{Proposition}
\theoremstyle{remark}
\newtheorem{remark}[theorem]{Remark}

\newtheorem{definition}[theorem]{Definition}

\makeatletter
\@namedef{subjclassname@2010}{%
	\textup{2010} Mathematics Subject Classification}
\makeatother
\begin{document}
	
	\title{a domain with non-plurisubharmonic $d$-balanced squeezing function }
	
	\author[N. Gupta]{Naveen Gupta}
	\address{Department of Mathematics, University of Delhi,
		Delhi--110 007, India}
	\email{ssguptanaveen@gmail.com}
	
	
	
	\begin{abstract}
In this note, we give an example of a domain whose $d$-balanced squeezing function is non-plurisubharmonic.
\end{abstract}

\maketitle
	\section{introduction}
	We present an example of a domain in $\mathbb{C}^2$ whose $d$-balanced squeezing function fails to be plurisubharmonic. Let us first recall some related notions.

	For a bounded domain $D\subseteq \mathbb{C}^n$ and $z\in{D}$, Deng et al \cite{2012}  introduced  the squeezing function on $D$, denoted by $S_{D}$,  as follows:
	$$S_{D}(z):=\sup_f\{r:\mathbb{B}^n(0,r)\subseteq f(D), f\in{\mathcal{O}_u(D,\mathbb{B}^n)}\},$$ 
	where $\mathbb{B}^n(0,r)$ denotes a ball of radius $r$ centered at the origin and $\mathcal{O}_u(D,\mathbb{B}^n)$ denotes the collection of injective holomorphic maps from $D$ to unit ball $\mathbb{B}^n$.
	\smallskip
	
	Rong and Yang  \cite{gen}, extended this idea by replacing the unit ball with a bounded, balanced, convex domain. 
 Recall that a domain $\Omega\subseteq \mathbb{C}^n $ is called balanced if  $\lambda z$ belongs to ${\Omega} $
	for each $z$ in ${\Omega}$ and $\lambda$ belongs to  the closed unit disc  $\overline{\mathbb{D}}$ of the complex plane.  Its
Minkowski function $h_{\Omega}$  on 
	$\mathbb{C}^n$ is defined as 
	$$h_{\Omega}(z):=\inf \{t>0:z/t\in{\Omega}\}.$$
	For 
	a bounded domain $D\subseteq \mathbb{C}^n$ the generalized squeezing
	function $S_D^{\Omega}$ on $D$ is defined as
	\begin{equation*}\label{eqn:gensq}
		S^{\Omega}_D(z):=\sup \{r:\Omega(r)\subseteq f(D), f\in{\mathcal{O}_u(D,\Omega)}, f(z)=0\}.
	\end{equation*} 
	
	In \cite{d-balanced}, we introduced the $d$-balanced squeezing function by replacing a balanced domain by a $d$-balanced domain. .
	
	Let 
	$d=(d_1,d_2,\ldots,d_n)\in{\mathbb{Z}_n^{+}},n\geq 2$. Then a domain $\Omega\subseteq \mathbb{C}^n$
	is said to be $d$-balanced if for each $z=(z_1,z_2,\ldots, z_n)\in \Omega$
	and $\lambda \in{\overline{\mathbb{D}}}$, $\left(\lambda^{d_1}z_1,\lambda^{d_2}z_2,\ldots, 
	\lambda^{d_n}z_n\right) \in \Omega$.
	
	For a $d$-balanced domain $\Omega$,  the $d$-Minkowski function on $\mathbb{C}^n$,
is	denoted by $h_{d,\Omega}$ and is defined as 
	$$h_{d,\Omega}(z):=\inf \{t>0:\left(\frac{z_1}{t^{d_1}},\frac{z_2}{t^{d_2}},\ldots,
	\frac{z_n}{t^{d_n}}\right)\in{\Omega}\}.$$
	
	\begin{definition} 
		For a bounded domain $D\subseteq \mathbb{C}^n$, and a bounded, convex, $d=(d_1,d_2,\ldots, d_n)$-balanced domain $\Omega$, 
		\emph{ the $d$-balanced squeezing function} (denoted by $S_{d,D}^\Omega$)
		is given by:
		\begin{equation*}\label{eqn:gensq}
			S_{d,D}^\Omega(z):=\sup \{r:\Omega^d(r)\subseteq f(D), f\in{\mathcal{O}_u(D,\Omega)}, f(z)=0\}.
		\end{equation*}
	\end{definition}
	We can easily see that if $\Omega$ is balanced, $d=(1,1,\ldots, 1)$, therefore $S_{d,D}^{\Omega}$ reduces to $S_D^{\Omega}.$
	
	In \cite{nik-psh}, Forn\ae ss and Scherbina gave an example of a domain whose squeezing function is non-plurisubharmonic. Recently, Rong and Yang \cite{feng-rec},  gave examples of domains with non-plurisubharmonic  generalized squeezing functions. Here we consider the same problem for $d$-balanced squeezing functions and present an example (see Theorem \ref{thm:main}).
	
	\section{Background and an estimate of $d$-balanced squeezing function}
	
 Let us first recall the definitions of the Carathéodory pseudodistance and the Carathéodory extremal maps. 
	For a domain  $D\subseteq \mathbb{C}^n$ and $z_1,z_2\in D$, the Carathéodory pseudodistance  $c_{D}$ on $D$ is defined as 
	$$c_D(z_1,z_2)=\sup_f \{p(0,\mu):f\in{\mathcal{O}(\mathbb{D}, D), f(z_1)=0, f(z_2)=\mu}\},$$
	where $p$ denotes the Poincar\'e metric on unit disc $\mathbb{D}$. A function $f\in {\mathcal{O}(\mathbb{D}, D)} $ at which this supremum is attained is called the Carathéodory extremal function.  
	
	We now recall a few results that will be used in this section. Note that Lempert \cite[~Theorem 1]{lempert-classic}, Kosi\`{n}ski et al \cite[~Theorem 1.3]{lempert}
	and the Remark 1.6 therein yields the following.
	\begin{result}\label{res:lempert}
		For a convex domain $\Omega\subseteq \mathbb{C}^n,\ c_{\Omega}
		=\tilde{k}_{\Omega},$
		where $\tilde{k}_{\Omega}$ denotes the Lempert function on $\Omega$.
	\end{result}
Combining Result \ref{res:lempert} with \cite[~Theorem 1.6]{bharali}, we get the following result:
	
	\begin{result}\label{res:bharali}
		For a bounded, convex, $d$=$(d_1,d_2,\ldots,d_n)$-balanced domain $\Omega\subseteq \mathbb{C}^n$, 
		$$\tanh^{-1}h_{d,\Omega}(z)^L\leq c_{\Omega}(0,z)=\tilde{k}_{\Omega}(0,z)\leq \tanh^{-1}h_{d,\Omega}(z),$$
		where $L=\max_{1\leq i\leq n}d_i.$
	\end{result}
	
	\begin{result}\label{lem:extension}
		Let $\Omega\subseteq \mathbb{C}^n$ be a domain and $K\subseteq \Omega$ be compact such that $\Omega \setminus K$ is connected. Then, each holomorphic fucntion $f$ on $\Omega \setminus K$ extends to a holomorphic function on $\Omega$.
	\end{result}
\begin{result}[see Proposition 1 in\cite{nikolov-d}] \label{prop:psh}
	Let $\Omega\subseteq \mathbb{C}^n$ be a balanced domain and let $h_{d,\Omega}$ be its d-Minkowski function. Then $\Omega$ is pseudoconvex if and only if $h_{d,\Omega}$ is plurisubharmonic.
\end{result}
\begin{result}\label{res:basicminko}
	For a $d$-balanced domain $\Omega\subseteq \mathbb{C}^n$, the following holds:
	\rm{( see\cite[~Remark 2.2.14]{pflug})}
	\begin{enumerate}
		\item $\Omega =\{z\in{\mathbb{C}^n}:h_{d,\Omega}(z)<1\}$.
		\item $h_{d,\Omega}\left(\lambda^{d_1}z_1,\lambda^{d_2}z_2,\ldots, \lambda^{d_n}
		z_n\right)=|\lambda |h_{d,\Omega}(z)$ for each $z=(z_1,z_2,\ldots, z_n)\in{\mathbb{C}^n}$
		and $\lambda\in{\mathbb{C}}.$
		\item $h_{d,\Omega}$ is upper semicontinuous.
	\end{enumerate}
\end{result}

For a bounded domain $\Omega\subseteq \mathbb{C}^n$ and a compact subset $K$ of $\Omega$, denote
$$d_{c_{\Omega}}^{K}(z)=\min_{w\in K}\tanh (c_{\Omega}(z,w)). $$
	
	We begin with the following theorem for $d$-balanced domains, which is analogues to Theorem 2.1 in \cite{feng-rec}.

\begin{theorem}\label{thm:rel}
	Let $\Omega\subseteq \mathbb{C}^n$ be a bounded, $d$=$(d_1,d_2,\ldots, d_n)$-balanced, convex, homogeneous domain.
	 If $K$ is a compact subset of $\Omega$ such that $D=\Omega \setminus K$ is connected, then
	\begin{equation}\label{eqn:rel}
	S^d(z)^L\leq d_{c_{\Omega}}^{\partial K}(z)=d_{c_{\Omega}}^{K}(z)\leq S^d(z),
	\end{equation}
	where $L=\max_{1\leq i \leq n}d_i.$
\end{theorem}
\begin{proof}
	For $z\in D$, let $g\in{Aut(\Omega)}$ be such that $g(z)=0.$ Because of the onvexity of $\Omega$,
	$\{tv+(1-t)w:0\leq t\leq 1,\, v\in{D}, w\in K^{\circ}\}\cap \partial K\neq \emptyset,$  therefore, 
	$d_{C_{\Omega}}^{\partial K}(v)=d_{C_{\Omega}}^{K}(v)$ for each $v\in {D}.$
	
 Clearly, $h=g|_{D}:D\to \Omega$ is injective holomorphic with $h(z)=0$. For notational convenience, let us denote 
 $d_{C_{\Omega}}^{K}(z)$ by $\alpha$. We claim that $\Omega^d(\alpha)\subseteq h(D)$. Let $h_{d,\Omega}(v)<\alpha $, which upon using Result \ref{res:bharali} implies $\tanh (c_{\Omega}(0,v))<\alpha $. Since $g$ is an automorphism, we get $\tanh (c_{g(\Omega)}(g(z),g(v')))<\alpha$, $v'\in \Omega$. Therefore, $\tanh (c_{\Omega}(z,v'))<\alpha=\min_{w\in K}\tanh (c_{\Omega}(z,w))$. Thus, we get $v'\notin K$ and therefore, $v=g(v')\in g(D)$. This proves our claim and hence, we obtain $$S^d(z)\geq \alpha =d_{C_{\Omega}}^{K}(z).$$
 
 For the other side inequality, consider an injective holomorphic map $f:D\to \Omega$ such that $f(z)=0$. By Result \ref{lem:extension}, there exists a holomorphic function $F:\Omega \to \mathbb{C}^n$ such that $F|_{D}=f$. Using Result \ref{prop:psh} and following the argument as in \cite[~Theorem 2.1]{feng-rec}, we obtain $F(\Omega)\subseteq \Omega$. Observe that $F(\partial K)\cap F(D)\neq \emptyset$. Let $r>0$ be such that $\Omega^d(r)\subseteq F(D)$. If possible, let $\tanh (c_{\Omega}(0,F( \partial K)))^{1/L}<r$, then upon using Result \ref{res:bharali}, we get $h_{d,\Omega}(F(\partial  K))<r.$ This implies that $F(\partial  K)\in \Omega^d(r)\subseteq F(D)$, which is a contradiction. Therefore $r<\tanh (c_{\Omega}(0,F(\partial  K)))^{1/L}$, which upon using the decreasing property of $C_{\Omega}$ implies that $r<\tanh (c_{\Omega}(z, \partial K))^{1/L}$. Finally we can conclude that $	S^d(z)^L\leq d_{c_{\Omega}}^{ K}(z).$ \qedhere
\end{proof}

\begin{remark} \label{rem:observation}
	A careful look at the above proof makes it clear that the left hand side of inequality 
\ref{eqn:rel} holds even if $\Omega$ is not homogeneous.
\end{remark}

\section{Non-plurisubharmonic $d$-balanced squeezing function}

Let $G_2\subseteq \mathbb{C}^2$ be the domain defined by 
$$G_2=\{(z_1+z_2,z_1z_2):z_1,z_2\in \mathbb{D}\},$$ where $\mathbb{D}$ denotes the unit disc in $\mathbb{C}.$ The domain $G_2$ is called the symmetrized bidisc. Its genesis lies with the problem of `robust stabalization' in control engineering. Althought it is closely related to bidisc, its geometry is very different from the bidisc. This is polynomially convex, hyperconvex and starlike about the origin but not convex ($(2,1),\, (2i, -1)\ G_2$ but $(1+i,0)\notin G_2$). Another point to note here is that it is not homogeneous (there is no automorphism of $G_2$, which maps for any $0<a<1$, $(a,0)$ to $(0,0)$). For many equivalent characterisations of $G_2$, \cite{agler} can be referred to.  

It should be noted here that  domain $G_2$ has several interesting properties. For example, Lempert's theorem holds for $G_2$ even though it is neither convex nor  can it be exhausted by domains biholomorphic to convex domains. 

We will  require Carathéodory extremal maps for $G_2$ to prove our result. Agler and Young, in their paper \cite{agler}, proved that for each $z_1,z_2\in G_2$, there exists $\lambda\in \mathbb{C}, |\lambda|=1$ such that $\phi_{\lambda}$ is the Carathéodory extremal function, where $\phi_{\lambda}$ is defined as $$\phi(z_1,z_2)=\frac{2\lambda z_2-z_1}{2-\lambda z_2}.$$

It is easy to check that $G_2$ is $(1,2)$-balanced. Let us denote by $\Omega$ the set of all possible linear combinations of elements of $G_2$, that is, $\Omega$ is the convex hull of $G_2$. We begin with the following lemma.
\begin{lemma}\label{lem:(1,2)}
	 Let $\Omega$ be convex hull of $G_2$. Then, the following holds:
\begin{enumerate}
		\item $\Omega$ is $(1,2)$-balanced.
		\item For each $(z_1, z_2)\in \Omega$, $|z_1|<2$ and $|z_2|<1$.
	\end{enumerate}
\begin{proof} \mbox{} 
\begin{enumerate}
	\item Let $\sum_{i=1}^k\alpha_i z_i\in \Omega$, $\sum_{i=1}^k\alpha_i=1,\ \alpha_i\geq 0$  and $z_i=(z_i^{(1)},z_i^{(2)})\in G_2$ for each $i.$ Let $|\lambda |\leq 1$. Each $z_i\in G_2$; therefore, $(\lambda z_i^{(1)},\lambda^2 z_i^{(2)})\in G_2$. Thus, $\sum_{i=1}^k\alpha_i (\lambda z_i^{(1)},\lambda^2 z_i^{(2)})=\left(\lambda \sum_{i=1}^k\alpha_i z_i^{(1)},\lambda^2 \sum_{i=1}^k\alpha_i z_i^{(2)}\right)\in \Omega$ and hence, $\Omega$ is $(1,2)$-balanced. 
	\item Follows from the structure of $G_2$.
	\end{enumerate}
\end{proof}
\end{lemma}

Choose $0<r<1$ such that closure of the polydisk in $\mathbb{C}^2$ of radius $r$ centered at the origin, denoted by $\overline{\mathbb{D}^2(0,r)}$, is contained in $\Omega$. Take $Q=(0,r)\in \overline{\mathbb{D}^2(0,r)} \subseteq  \Omega$ and let $\epsilon >0$ be such that a ball of radius $\epsilon<r$ centered at $Q$, denoted by $\mathbb{B}^2(Q,r)$, is contained in  $\Omega$. Let us take $K=\partial \mathbb{D}^2(0,r)\setminus \mathbb{B}^2(Q,\epsilon)$. It can be seen that $K$ is compact and $D=\Omega \setminus K$ is connected. We will show that $S_{d,D}^{\Omega}$ (denoted by $S^d$ for notational convenience) is not plurisubharmonic. For this, we will show that $S^d|_{\mathbb{D}^n(0,r)\cap H}$ does not satisfy the maximum principle. 
In particular, such a restriction is not subharmonic; this, in turn, will imply that $S^d$ is not plurisubharmonic. We will prove this via the following steps:

\begin{itemize}
	\item We begin by showing that $S^d(0)\geq \frac 1 2$.
	\item We use Theorem \ref{thm:rel} to show that $S^d(z)^2\leq \dfrac{r-|z|}{2-r|z|} $ for $z\in \mathbb{D}^n(0,r)\cap H,\ z\neq 0$ (observe Remark \ref{rem:observation}).
	\item We then show that $S^d(z)\leq S^d(0)$ for $z=(z_1,0)\in \mathbb{D}^n(0,r)\cap H$ for $r>|z|>\beta$, where $\beta=\frac{r(4-r)}{4-r^3}$ (note that $\beta<r$) using some calculations.
	\item Now we restrict $h$ to $A=\overline{\mathbb{B}(0,\beta)}$ to obtain a maximum at some $a\in A $. 
	\item We conclude by combining all these points along with the observation that $S^d(z)\to 0$ as $|z|\to r$.
\end{itemize}

\begin{lemma}\label{lem:0}
	$S^d(0)\geq \frac 1 2$.
\end{lemma}
\begin{proof}
	Consider the identity map $\mathbbm{id}:D\to \Omega$. Clearly,  $\mathbbm{id}$ is injective holomorphic with $\mathbbm{id}(0)=0$. We claim that $\Omega^d(r/2)\subseteq \mathbbm{id}(D)=D$. To see this, take $z$ such that $h_{d,\Omega}(z)<r/2.$ Upon using  Result \ref{res:basicminko}, we first get $z\in \Omega$ and that $\left(\frac{2z_1}{r}, \frac{4z_2}{r^2}\right)\in \Omega.$ Now using Lemma \ref{lem:(1,2)}(2), we get $|z_1|<r$ and $|z_2|<\frac{r^2}{4}<r$. Thus, $z\notin \partial \mathbb{D}^2(0,r)$ and therefore, $z\notin K.$ Hence we prove our claim and this shows that $$S^d(0)\geq \frac 1 2. $$\qedhere \end{proof}
We need the following elementary lemma to prove our next proposition.
\begin{lemma}\label{lem:z}
For each $z=(z_1,0)$, $z_1=a+ib$ with $0<|z_1|<r$ and $w_0=(sz_1,0)$, where $s=\frac{r}{\sqrt{a^2+b^2}}$, the following holds:
$$\left|\frac{\phi_{\tau}(z)-\phi_{\tau}(w_0)}{1-\phi_{\tau}(z)\phi_{\tau}(w_0)}\right| \leq \dfrac{r-|z_1|}{1-r|z_1|},$$
where  $\tau\in \mathbb{C}$ with $|\tau|=1$. 
\end{lemma}
\begin{proof}
First note that $|w_0|=r$ and $|(2-\tau z_1)(2-\tau w_0)|\geq (2-r)^2>1$. Now consider 
\begin{align*}
	\left|\frac{\phi_{\tau}(z)-\phi_{\tau}(w_0)}{1-\phi_{\tau}(z)\phi_{\tau}(w_0)}\right|&=\left|\frac{\frac{z_1}{2-\tau z_1}-\frac{w_0}{2-\tau w_0}}{1-\frac{z_1}{2-\tau z_1}\frac{w_0}{2-\tau w_0}}\right|\\
	&=\left|\frac{\frac{2(z_1-w_0)}{(2-\tau z_1)(2-\tau w_0)}}{1-\frac{z_1}{2-\tau z_1}\frac{w_0}{2-\tau w_0}}\right|\\
	&\leq \frac{2|z_1-w_0|}{1-|z_1||w_0|}\\
	&=\frac{r-|z_1|}{1-|z_1|r} \qedhere
\end{align*}
\end{proof}
\begin{proposition}\label{prop:z}
	For each $z=(z_1,0)\in \mathbb{D}^n(0,r)\cap H,\ z\neq 0$, we have $$S^d(z)^2\leq \dfrac{r-|z_1|}{1-r|z_1|}.$$
\end{proposition}
\begin{proof}
	Let $z_1=a+ib$. Take $w_0=(sz_1,0)$, where $s=\frac{r}{\sqrt{a^2+b^2}}$ so that $w\in K$  because $Q<r.$
	Now, using Theorem \ref{thm:rel} and Remark \ref{rem:observation}, we obtain 
	\begin{align*}
	S^d(z)^2&\leq d_{c_{\Omega}}^K(z)\\
	&=\min_{w\in K}\tanh (c_{\Omega}(z,w))\\
	&\leq \tanh (c_{\Omega}(z,w_0))\\
    &\leq \tanh (c_{G_2}(z,w_0))\ \ \ \ \ \ \ \ \ \ \ \ \ \ \ \ \ \ \ \ \ \ \ \ \ \ \ \ \ \ \ \ \ \ \ \ \ \ \ \ \ \ \ \ \ \  (\mbox{since $G_2\subseteq \Omega$}) \\
	&=\left|\frac{\phi_{\tau}(z)-\phi_{\tau}(w_0)}{1-\phi_{\tau}(z)\phi_{\tau}(w_0)}\right| \ \ \ \ \ \ \ \ \ \ \ \ \ \ \ \ \ \ \ \ \ \ \ \ \ \ \ \ \ \ \ \ \ \ \ \ \ \ \ \ \ (\mbox{for some $|\tau|\leq 1$})\\
	&\leq \dfrac{r-|z_1|}{1-r|z_1|}.\ \ \ \ \ \ \ \ \ \ \ \ \ \ \ \ \  \ \ \ \ \ \ \ \ \ \ \ \ \ \ \ \ \ \ \ \ \ \ \ \ \ \ \ \ \ \ \ \ \ \ \ \ (\mbox{using Lemma \ref{lem:z}}) \qedhere
\end{align*}
\end{proof}
This proposition, in particular, implies $S^d(z)\to 0$ as $|z_1|\to r$. We summarise these results with the following theorem.
\begin{theorem} \label{thm:main} 
For $D$ and $\Omega$ as considered above, $S^d$ is not plurisubharmonic.
\end{theorem}
\begin{proof} 
Note that $\dfrac{r-|z_1|}{1-r|z_1|}<\dfrac{r^2}{4}$ if and only if $|z_1|>\beta$. Therefore, using Proposition \ref{prop:z} and Lemma \ref{lem:0}, we get $h(z)=S^d(z)<\dfrac{r}{2}\leq S^d(0)=h(0)$ for $z=(z_1,0)$ with $|z_1|>\beta$. Consider the restriction  $h|_{\overline{\mathbb{B}(0,\beta)}}$ and let $h(a)$ be its maximum for some $a\in{\mathbb{B}(0,\beta)}.$ Then, $h(z)<\max(h(a),h(0)),$ proving that $S^d|_{\mathbb{D}^n(0,r)\cap H}$  does not satisfy the maximum princple and $S^d$ is not plursubharmonic.
\end{proof}
\medskip

\section*{Acknowledgement} I thank my thesis advisor Sanjay Kumar for reading the manuscript and suggesting changes.

\end{document}